\documentclass[12pt]{article}
\usepackage{latexsym}
\usepackage{amssymb}
\usepackage{amsmath}
\usepackage{amsthm}
\usepackage{rotating}
\usepackage{multirow}
\usepackage{mathrsfs}
\usepackage{wasysym}
\usepackage{slashed}
\usepackage{rawfonts}
\input{prepictex}
\input{pictex}
\input{postpictex}
\DeclareMathAlphabet{\zap}{OT1}{pzc}{m}{it}
\def\CC{\mathbb C}
\def\drc{\slashed{D}}
\newtheorem{main}{Theorem}

\DeclareMathOperator{\Aut}{Aut}
\DeclareMathOperator{\Iso}{Iso}

\newtheorem{thm}{Theorem}
\newtheorem*{them}{Theorem}
\newtheorem{lem}{Lemma}
\newtheorem{prop}{Proposition}

\newtheorem{*corem}{Corollary}
\newtheorem{defn}{Definition}
\newtheorem*{conj}{Conjecture}
\newenvironment{xpl}{\mbox{ }\\ {\bf  Example.}\mbox{ }}{
\hfill $\diamondsuit$\mbox{}\bigskip}
\newenvironment{rmk}{\mbox{ }\\{\bf  Remark.}\mbox{ }}{
\hfill $\diamondsuit$\mbox{}\bigskip}
\def\ZZ{{\mathbb Z}}
\def\RR{{\mathbb R}}
\def\CP{{\mathbb C \mathbb P}}
\def\thor{\mbox{\thorn}}
\begin{document}

\title{Weyl Curvature, Del Pezzo Surfaces, and  Almost-K\"ahler Geometry}

\author{Claude LeBrun\thanks{Supported 
in part by  NSF grant DMS-1205953.}\\Stony
 Brook University}

\date{}
\maketitle

 \begin{abstract}	
If  a smooth compact $4$-manifold $M$ admits a K\"ahler-Einstein metric
 $g$ of positive scalar curvature, Gursky \cite{gursky} showed that its conformal class 
 $[g]$ is an absolute minimizer of the Weyl functional among all conformal
 classes with positive Yamabe constant.  Here we prove  that, with the same
 hypotheses, $[g]$ also  minimizes of the Weyl functional on a different open set of 
 conformal classes, most of which have {\em negative} Yamabe constant.  
 An analogous minimization result 
 is then proved for Einstein metrics $g$ which are Hermitian, but not K\"ahler. 
\end{abstract}

 The 
curvature tensor $\mathcal{R}$ of a smooth Riemanian $n$-manifold $(M,g)$ 
 can  invariantly be decomposed into 
the scalar curvature $s$, the 
trace-free Ricci curvature $\mathring{r}$, 
and the  Weyl curvature $W$:
$${\mathcal{R}^{ab}}_{cd}= {W^{ab}}_{cd} + \frac{4}{n-2} \mathring{r}^{[a}_{[c}\delta^{b]}_{d]}+\frac{2}{n(n-1)} ~s~\delta^a_{[c}\delta^b_{d]}.$$
The Weyl curvature tensor is exactly the conformally invariant part of $\mathcal{R}$,
in the sense that 
 ${W^a}_{bcd}$ is unchanged if we multiply  $g$ by any smooth 
smooth positive function, while $\mathring{r}$ and $s$   can both 
be made to vanish at any given point by making a suitable choice of conformal factor. 
Moreover, if $n\geq 4$, a Riemannian metric is locally conformally flat iff 
its Weyl tensor vanishes identically.  

If $M$ is a smooth compact oriented $n$-manifold,  $n\geq 4$, 
the 
{\em Weyl functional}
$$
\mathscr{W}([g ]) = \int_M |W_g|^{n/2}d\mu_g
$$
then only depends on the conformal class
$$
[g] = \{ u^2 g~|~ u: M \stackrel{C^\infty}{\to} \RR^+\}
$$
of the metric, and may 
be considered as  a natural measure of how far $[g]$  deviates from conformal flatness. 
It is thus natural and interesting to 
study the infimum of $\mathscr{W}$ among all metrics on a given manifold $M$, and to 
ask  whether there is actually a minimizing metric which achieves this infimum. 
This problem seems to have first
been discussed by  Atiyah, Hitchin, and Singer \cite{AHS}, who observed that, when $n=4$,
\begin{equation}
\label{sig}
\mathscr{W}([g ]) \geq 12\pi^2 |\tau (M)|
\end{equation}
where $\tau (M) = b_+(M)-b_-(M)=p_1(M)/3$ denotes the signature
 of $M$. 
 Soon afterwards, Osamu Kobayashi \cite{okobweyl} went on to   observe  that
 this  generalizes to other dimensions: whenever 
 the oriented cobordism class  $[M]$ has infinite order,  
   $\inf \mathscr{W}$ must be strictly positive. 
 
 The $4$-dimensional case  features many idiosyncracies which 
 make the problem particularly important and compelling in this dimension. For example, if 
 $g$ is an Einstein metric on a smooth compact $4$-manifold $M$,  its conformal 
 class $[g]$ is a critical point of the Weyl functional; but  this is not true
 in higher dimensions, as  the Riemannian product of the
 unit $2$-sphere with the round $m$-sphere of radius $\sqrt{m-1} \neq  1$ is an Einstein
 manifold that does 
 not solve the Euler-Lagrange equations for $\mathscr{W}$. Generally speaking, 
 the special 
 character of $n=4$ arises from the fact that there is an invariant decomposition 
 $$\Lambda^2 = \Lambda^+\oplus \Lambda^-,$$
 depending only on the orientation and the conformal class $[g]$, 
 of  the $2$-forms into their self-dual and anti-self-dual parts.
  This in turn induces a conformally invariant decomposition
 $$W= W_++W_-$$
 of the Weyl tensor into the self-dual and anti-self dual Weyl tensors
  $W_\pm$,
 explicitly given by 
 $${(W_\pm)^a}_{bcd} = {\textstyle \frac{1}{2}} {W^a}_{bcd} \pm {\textstyle \frac{1}{4}} {(d\mu)_{cd}}^{ef}{W^a}_{bef}.$$
 Using this, the $4$-dimensional Thom-Hirzebruch signature formula $\tau = p_1/3$
 can be written as 
 $$
 \tau (M) = \frac{1}{12\pi^2} \int_M \left(|W_+|_g^2 - |W_-|_g^2 \right) d\mu_g 
 $$
 for any metric $g$ on $M$. 
Since 
 $$\mathscr{W}([g ]) =  \int_M \left(|W_+|_g^2 + |W_-|_g^2 \right) d\mu_g ~;$$
 we immediately deduce  \eqref{sig}, with equality iff  $W_-= 0$ (in which case $[g]$ is
 said to be {\em self-dual}) 
 or 
 $W_+= 0$ (in which case $[g]$ is
 said to be {\em anti-self-dual}). The same reasoning  yields the identity 
 $$
 \mathscr{W}([g ]) =  - 12\pi^2 \tau (M) + 2  \int_M |W_+|_g^2  d\mu_g, 
 $$
 so that the study of the Weyl functional on a fixed $4$-manifold is completely equivalent to
 studying the $L^2$-norm of the self-dual Weyl curvature.

 One remarkable consequence   \cite{AHS}  is that   the
 Fubini-Study metric on $\CP_2$ is an absolute minimizer of the Weyl functional; moreover, 
 Poon \cite{poonthesis} proved that, up to conformal isometry, this is 
 the {\em only}  minimizer  on $\CP_2$   with  positive Yamabe constant. 
 While these arguments hinge on the fact that the Fubini-Study metric is self-dual, one might wonder 
 whether  this might somehow also reflect the fact that it  is  a {\em K\"ahler-Einstein} metric. For example, 
 the standard product metric on $S^2\times S^2$ is  also K\"ahler-Einstein, and 
 Kobayashi \cite{okobweyl} discovered interesting evidence in support of the conjecture that
 this metric is also a minimizer of the Weyl functional. In fact,  an elegant argument 
 due to Gursky \cite{gursky} provides a beautiful, general  result in this direction:

 \begin{them}[Gursky]
 Let $M$ be a smooth compact oriented $4$-manifold with $b_+\neq 0$. Then 
 any conformal class $[g]$ on $M$ with Yamabe constant $Y([g]) > 0$  satisfies 
 $$\int_M |W_+|^2d\mu \geq \frac{4\pi^2}{3} (2\chi + 3\tau ) (M),$$
 with equality iff $[g]$ contains a K\"ahler-Einstein metric $g$ with $\lambda >0$. 
 \end{them}

 Of course, this does not quite answer the question, because the requirement that $[g]$ have positive
 Yamabe constant excludes ``most'' conformal classes on $M$. In this article, we will 
push our understanding of the problem in a new direction by
 showing that  
 $\lambda > 0$  K\"ahler-Einstein metrics also minimize the
 Weyl functional for  a different 
 large open set of conformal classes, most of which have {\em negative}
 Yamabe constant. We will then extend this result to Einstein metrics which are merely Hermitian, rather than K\"ahler.  
 
 If a compact complex surface $(M,J)$ admits an Einstein metric $g$ which is Hermitian
 with respect to $J$ and has Einstein constant $\lambda > 0$, then $(M,J)$ has \cite{lebhem}
 ample anti-canonical line bundle $K^{-1}$,  often abbreviated as  $c_1 >0$.
 Complex surfaces
 with $c_1 >0$ are called {\em del Pezzo surfaces} \cite{delpezzo,cubic}; 
 they are precisely the Fano manifolds
 of complex dimension $2$. Every del Pezzo surface conversely admits 
 \cite{chenlebweb,lebhem10,sunspot,tian} a Hermitian, $\lambda >0$  Einstein metric which is compatible with the specified complex structure, and this metric is  unique \cite{lebuniq}
 up to complex automorphisms and rescalings. Any del Pezzo surface is biholomorphic  either to 
 $\CP_1 \times \CP_1$ or to a blow-up $\CP_2\# k \overline{\CP}_2$ of the complex projective plane at $k$ points in general position,  $0\leq k \leq 8$. In most cases, the relevant Einstein metric is actually K\"ahler-Einstein. In fact, there are just two cases
 in which this fails to be true: $\CP_2\#  \overline{\CP}_2$ and $\CP_2\# 2 \overline{\CP}_2$.
 In these exceptional  cases, the Einstein metric is not K\"ahler, but is nonetheless  related to a
 K\"ahler metric by conformal rescaling. 
 
 One important topological property of any del Pezzo surface  $(M,J)$ is that $b_+(M)=1$, so that 
the intersection form
 \begin{eqnarray*}
 H^2(M,\RR) \times H^2(M,\RR)&\longrightarrow&\quad \RR\\
 (\quad [\phi ]\quad , \quad [\psi ]\quad ) \quad &\longmapsto & \int_M \phi \wedge \psi
\end{eqnarray*}
is a Lorentzian inner product. Consequently, the space of self-dual harmonic
$2$-forms is  $1$-dimensional for any Riemannian metric $g$   on a del Pezzo surface
 $M$. However, since  a 
self-dual $2$-form is harmonic iff it is closed, the space of 
self-dual 
$2$-forms depends only on the conformal class $[g]$ of the metric.
Thus, up to multiplication by a non-zero real constant,
there is a unique non-trivial self-dual harmonic $2$-form $\omega$ associated with any 
conformal class $[ g]$. We will say that $[g ]$ is of {\em symplectic type} if this harmonic
self-dual $2$-form satisfies $\omega \neq 0$ at every point of  $M$. When this happens, 
$(M, \omega )$ is a symplectic $4$-manifold, and in particular has a well-defined
first Chern class $c_1 = c_1(M,\omega )$.
   
 \begin{main} \label{cornerstone}
 Let $M$ be the underlying $4$-manifold of a del Pezzo surface.
 Then any conformal class $[g]$ of symplectic type on $M$ satisfies 
 $$\int_M |W_+|^2 d\mu \geq \frac{4\pi^2}{3} \frac{(c_1\cdot [\omega ])^2}{[\omega ]^2},$$
 with equality iff $[g]$ contains a K\"ahler metric $g$ of constant scalar curvature.
 \end{main}
 
The proof of this result is presented in \S \ref{rrs} below, where we also observe that it immediately implies  an  interesting  variant of Gursky's Theorem:
 
 \begin{main} \label{keystone}
  Let $M$ be the underlying smooth oriented $4$-manifold of a del Pezzo surface.
 Then 
  any conformal class $[g]$ of symplectic type on $M$ satisfies 
  $$\int_M |W_+|^2 d\mu \geq \frac{4\pi^2}{3} (2\chi + 3\tau )  (M) ,$$
 with equality iff $[g]$ contains a K\"ahler-Einstein metric $g$.
 \end{main}
 
 However, the lower bound occurring in Theorem \ref{cornerstone} is stronger than 
 that in Theorem \ref{keystone}. As we will see in \S \ref{period}, this 
 leads to some interesting constraints on the geometry of Einstein metrics.

  The condition that a conformal class be of symplectic type is {\em open}  \cite{kohonda,lcp2} 
  in  the $C^{2,\alpha}$
  topology; in this sense, the condition is analogous to the Gursky's condition that the Yamabe constant  be positive. 
  Nonetheless,   ``most''  conformal classes of symplectic type have $Y([g]) < 0$;
  see \cite{jongsuak,jongsuak2}, 
  or Proposition \ref{neg} below. 
  Thus, Gursky's result  and Theorem \ref{cornerstone}  apply to rather different 
  open sets in the space of conformal classes. This might be interpreted as supporting the conjecture that,  in real dimension $4$,  positive K\"ahler-Einstein metrics are absolute minimizers of the Weyl functional among {\em all} conformal classes on the fixed manifold. 
  
 Gursky's result and Theorem \ref{cornerstone} provide us with a good picture of  
  the behavior of   the Weyl 
 functional  in the vicinity of an Einstein metric which is {\em K\"ahler}. 
  However, there are two del Pezzo surfaces, namely the one-  and two-point blow-ups of 
  $\CP_2$,  which {\em do not}
admit K\"ahler-Einstein metrics. In these cases, however, there  still  exist Einstein metrics 
which are {\em Hermitian} with respect to the complex structure. Moreover, such metrics 
are automatically conformally K\"ahler \cite{lebhem}, so their conformal classes are necessarily 
of symplectic type. It thus seems irresistible to ask whether these preferred Einstein metrics are minimizers in the same sense as their 
K\"ahler-Einstein cousins. 

Our present result  in this direction 
involves an additional assumption concerning  the 
conformal isometry group.  
The del Pezzo surfaces  $\CP_2\# \overline{\CP}_2$ and   $\CP_2\# 2\overline{\CP}_2$ are both 
{\em toric}, in the sense
that their complex automorphism groups contain a $2$-torus $T^2 = S^1 \times S^1$. Moreover, 
 any Einstein Hermitian metric on one of these manifolds is necessarily 
 toric, in the sense that it is isometric to a metric which is invariant under this fixed $2$-torus action. 
 This makes it natural to consider the conformal classes of these metrics as members of 
 the  symplectic conformal classes which are $T^2$-invariant. In this narrower context, we 
 can then show that  these special Einstein metrics do indeed minimize the Weyl functional:

 \begin{main} \label{capstone}
 Let $M$ be the underlying $4$-manifold of a toric del Pezzo surface, and let 
 $g$ be an Einstein metric on $M$ which is {\em Hermitian} 
  and invariant under the fixed torus action. 
 Then the conformal class $[g]$ minimizes the Weyl functional among symplectic conformal classes which 
 are invariant under the torus action. Moreover, up to  diffeomorphism, $[g]$  is 
 the  {\em unique} such minimizer. \end{main}
 
 The proof of this result is given in \S \ref{hermitian}.  In point of fact, the proof yields  considerably more information
 than is indicated by the mere statement of Theorem \ref{capstone},  and this has
  interesting consequences that  are  then 
 explored in \S \ref{period}. We  then conclude by discussing various questions and speculations
 in \S \ref{speculations}.

\section{Almost-K\"ahler Geometry}

Let $(M,\hat{g})$ be a smooth compact oriented Riemannian $4$-manifold ,
and let $\omega$ be a self-dual harmonic $2$-form on $M$. Suppose, moreover, 
that $\omega\neq 0$ at every point of $M$. Now, any self-dual $2$-form on 
an oriented Riemannian satisfies 
$$
\omega_{ab}\omega^{bc}= - {\textstyle \frac{1}{2}} |\omega|^2 \delta^c_a, ~~ \omega \wedge \omega = |\omega |^2 d\mu,
$$
 so that the closed $2$-form $\omega$ is, in particular,  a symplectic form 
on $M$. However, harmonicity is a conformally invariant condition on middle-dimensional 
forms, and we can obviously 
  conformally rescale the metric $\hat{g}$ to obtain a new 
   metric $g$ such that $|\omega |_{g}\equiv \sqrt{2}$;  the unique such 
${g}$ is explicitly given by 
$${g} =  {\textstyle\frac{1}{\sqrt{2}}}|\omega|_{\hat{g}} \hat{g}.$$
The endomorphism  $J$ of $TM$ then defined by ${J_a}^b=\omega_{ac}g^{cb}$ 
is  an almost-complex structure on $M$, and the relationship between $\omega$, $g$,
and $J$ is that of an almost-Hermitian manifold;  and because $\omega$ is a {\em  closed} $2$-form by hypothesis, one  says that $(M, g, \omega )$ is an {\em almost-K\"ahler manifold}
 \cite{yano}.  In this situation, it is traditional \cite{apodrag,blair} 
to define  the {\em star-scalar curvature} to be 
$$
s^* = 2{\mathcal R} (\omega , \omega )={\textstyle \frac{1}{2}} \mathcal{R}_{abcd}\omega^{ab}\omega^{cd}
$$
where ${\mathcal R}$ is the Riemann curvature tensor, so that, in our $4$-dimensional 
setting
\begin{equation} \label{mocha}
s^* = \frac{s}{3} + 2W_+(\omega , \omega ).
\end{equation}
Since the Weitzenb\"ock formula for self-dual harmonic 2-forms asserts
$$0=\frac{1}{2}\Delta |\omega|^2 + |\nabla \omega|^2 -2 W_+(\omega , \omega ) + \frac{s}{3}|\omega|^2,$$
the fact that $|\omega |^2=\frac{1}{2} \omega_{ab}\omega^{ab} \equiv 2$ implies that
\begin{equation}\label{java}
s^* -s = |\nabla \omega|^2
\end{equation}
and hence that $s^* \geq s$, with equality iff $(M,g,J)$  is a K\"ahler manifold. 

The average 
$$\varsigma = \frac{s^*+s}{2}$$
of the star-scalar and Riemannian scalar curvatures is usually called the 
{\em Hermitian scalar curvature} \cite{apodrag,lejmi} , and 
  will play a central role in our discussion. The importance of this quantity was perhaps
  first discovered by Blair \cite{blair}, although the true depth of its 
  significance has only recently emerged more recently, 
  through   Donaldson's clarification  \cite{donfield,dontor} of a moment-map picture 
first proposed by Fujiki \cite{biqdon,fujukisug}. For our immediate purposes, the
important point is that  the anti-canonical line bundle $L:=K^{-1}=\wedge^2 T^{1,0}$
admits a natural connection whose curvature $F$ has  self-dual part   given \cite{apodrag,lsymp}
by 
$$
F^+ =-i\left( \frac{\varsigma}{4}\omega + [W^+(\omega)]^\perp\right),
$$
where $[W^+(\omega)]^\perp$ is the component of the self-dual $2$-form $W^+(\omega )$
perpendicular to $\omega$. It follows that 
$$ {\textstyle \frac{1}{2}} \varsigma ~d\mu= i F^+ \wedge \omega = i F\wedge \omega $$
and hence that 
\begin{equation}
\label{blair}
\int_M \varsigma ~d\mu = 4\pi c_1 \cdot [\omega ].
\end{equation}
  
\begin{lem} \label{key}
Let $(M^4,g,\omega)$ be an almost K\"ahler manifold. Then $g$ satisfies 
\begin{equation}
\label{lemkey}
|W_+| \geq  \frac{\varsigma}{2\sqrt{6}} 
\end{equation}
with equality everywhere iff $g$ is K\"ahler, with K\"ahler form $\omega$, and scalar curvature $s\geq 0$. 
\end{lem}
\begin{proof}
Equations \eqref{mocha} and \eqref{java} tells us that 
$$W_+(\omega , \omega ) = \frac{s^*}{2} - \frac{s}{6} =  \frac{s^*+s}{6} + \frac{1}{3} |\nabla \omega |^2 \geq \frac{s^*+s}{6}= \frac{\varsigma}{3}~,$$
with equality only at points at which $\nabla \omega=0$.  On the other hand, 
if $$\lambda : M\to \RR$$ is the largest  eigenvalue of $W_+: \Lambda^+\to \Lambda^+$ at each point of $M$, we have
 $$
 W_+ (\omega , \omega ) \leq \lambda |\omega|^2 = 2\lambda,
 $$
 whereas 
 $$|W_+| \geq \sqrt{\frac{3}{2}} \lambda$$
 because $W_+$ is trace-free. Thus 
 $$
 |W_+| \geq \frac{1}{2} \sqrt{\frac{3}{2}} W_+ (\omega , \omega )  \geq   \frac{\varsigma}{2\sqrt{6}} ~,
 $$
 and equality can  occur only if $(M,g,\omega)$ is K\"ahler, with scalar curvature $\geq 0$. 
 
 Conversely, any 
 $4$-dimensional K\"ahler manifold satisfies $|W_+|= |s|/2\sqrt{6}$ and $s=\varsigma$. Such a 
 manifold therefore  saturates the inequality everywhere 
 iff its scalar curvature $s$ is non-negative. 
\end{proof}

Combining Lemma \ref{key} with \eqref{blair} now immediately yields 

\begin{lem} \label{lock}
Let $(M^4,g,\omega)$ be an almost K\"ahler manifold. Then $g$ satisfies 
\begin{equation}
\label{lemlock}
\int_M |W_+| ~d\mu \geq  \sqrt{\frac{2}{3}} ~\pi~c_1\cdot [\omega ]
\end{equation}
with equality  iff $g$ is K\"ahler, with K\"ahler form $\omega$, and scalar curvature $s\geq 0$. 
\end{lem}

\section{Rational and Ruled Surfaces} \label{rrs}

As was first observed by Taubes \cite{taubes2}, very few smooth compact $4$-manifolds admit symplectic structures
such that $c_1\cdot [\omega ] > 0$.  Indeed, using further results of Taubes \cite{taubes3},
Lalonde-McDuff  \cite{lalmcd} and  Liu \cite{liu1}   were later able to show that any such 
$4$-manifold is symplectomorphic to a rational or ruled complex surface, equipped with the
symplectic structure associated with  some choice of K\"ahler metric. Recall that 
a {\em minimal} ruled surface is  the total space of a holomorphic $\CP_1$-bundle
over a compact complex curve, and that, more generally,  a complex surface
is said to be {\em  ruled} iff it is an iterated blow-up of a minimal ruled surface. A complex surface
is said to be {\em rational} if it is obtained from $\CP_2$ by a sequence of blow-ups and
blow-downs. Thus, a surface is rational or ruled surface iff it is either $\CP_2$ or ruled. 
These are exactly   \cite{bpv,GH} the compact complex surfaces of K\"ahler type with Kodaira dimension $-\infty$. 
Every rational or ruled surface does admit K\"ahler classes such that $c_1\cdot [\omega ] > 0$, 
so the result of Lalonde-McDuff/Liu is sharp. 

If $M$ is the underlying smooth oriented $4$-manifold of a  rational or ruled complex surface,
then 
$b_+(M)=1$.  Consequently, every  conformal class $[g]$  on such an $M$ has an associated
self-dual harmonic $2$-form $\omega$, which is unique up to an overall multiplicative constant. 
In these circumstances, we will say that $[g]$ is of {\em symplectic type} if $\omega$ is everywhere
non-zero; this is an open condition on the conformal class. A choice of symplectic conformal class 
$[g]$, together
with a self-dual hamonic form $\omega$,
 determine an associated almost-complex structure $J$, and thus a 
first Chern class $c_1(M, \omega )$.  It should be emphasized that when 
we say that a conformal class of symplectic type satisfies $c_1\cdot
[\omega ] > 0$, the first Chern class $c_1$ appearing in this expression is always understood to denote $c_1(M, \omega )$; for example, if we replace $\omega$ with $-\omega$, we correspondingly  replace
$c_1$ with $-c_1$, so that the sign of  $c_1\cdot [\omega ]$ remains unchanged. Also note that while self-diffeomorphisms of $M$ can usually be used to produce many different connected components  in the space of conformal structures of symplectic type, Lalonde-McDuff/Liu guarantees that 
any two connected components are in fact related  by some diffeomorphism. In particular, we may 
choose to think of $c_1 (M, \omega )$ as the standard first Chern class of a fixed complex
structure, as long as we are willing to pay the price of sometimes moving a given conformal class 
$[g]$ by a diffeomorphism 
to put it in the connected component that contains the conformal classes of K\"ahler metrics.

 \begin{prop}\label{one}
Let $M$ be the underlying $4$-manifold of a rational or ruled complex surface, and let $[g]$ be a conformal class on $M$ of symplectic type, with 
associated symplectic form $\omega$. If $c_1\cdot [\omega ] \geq 0$, then the Weyl curvature of 
$[g]$ satisfies 
\begin{equation}
\label{proper}
 \int_M |W_+|^2 d\mu  \geq \frac{4\pi^2 }{3} \frac{ (c_1\cdot  [\omega])^2}{[\omega ]^2}
\end{equation}
with equality iff the conformal class $[g]$ contains a {\em K\"ahler metric} $g$ of constant, non-negative  scalar curvature. 
 \end{prop}

 \begin{proof}  Since $\int |W_+|^2 d\mu$ is conformally invariant, we may choose to conformally
 rescale our given metric  in any convenient  manner  
 before evaluating the integral. Since our hypotheses imply that the given metric
 is conformal to an almost-K\"ahler metric, we may therefore assume henceforth, without loss of generality,
 that $g$ is actually almost-K\"ahler.

However, Lemma \ref{lock} then tells us that this conformal representative $g$ satisfies 
 $$\int |W_+|~d\mu \geq  \sqrt{\frac{2}{3}} \pi ~c_1\cdot [\omega],$$
 with equality iff $g$ is K\"ahler with $s\geq 0$. Applying the Cauchy-Schwarz inequality, we therefore have 
$$ V^{1/2}\left( \int_M |W_+|^2 d\mu\right)^{1/2} \geq  \sqrt{\frac{2}{3}}\pi ~ c_1\cdot [\omega] ~,
$$
where $V= [\omega ]^2/2$ is the volume of the almost-K\"ahler manifold $(M,g,\omega )$; 
moreover,  equality 
would now also imply that $|W_+|$ is constant.  Since any K\"ahler metric satisfies
$|W_+|= |s|/2\sqrt{6}$, this shows that 
$$
\left( \int_M |W_+|^2 d\mu\right)^{1/2} \geq  \frac{2\pi~ c_1\cdot [\omega]}{\sqrt{3 ~[\omega ]^2}} 
$$
with equality  iff $g$ is a K\"ahler metric for which the scalar curvature $s$ is a non-negative constant. When the right-hand side is non-negative, 
squaring both sides then gives the promised result.  \end{proof}

For most  ruled surfaces, the sign of $c_1\cdot [\omega ]$ will depend on the choice of 
$[g]$. However,  this difficulty  does not occur when  $c_1^2 (M) \geq 0$.
While we could simply quote   Lalonde-McDuff \cite{lalmcd}
or Liu\cite{liu} to establish this, we will instead give a more   elementary proof
that some readers may find  illuminating. 

\begin{lem}
\label{limpid} 
Let $M$ be  the underlying $4$-manifold of a rational or ruled surface with 
$2\chi + 3\tau \geq 0$. Then any 
 conformal class $[g]$  of symplectic type on $M$ satisfies
 $c_1\cdot [\omega ]> 0$. 
\end{lem}
\begin{proof} 
Fixing the metric $g$ and its  compatible symplectic form $\omega$ also specifies  a spin$^c$ structure on $M$ with  Chern class
$c_1= c_1(M,\omega )$. Now any compact almost-complex $4$-manifold
satisfies $c_1^2 = 2\chi + 3\tau$, so our hypotheses guarantee that $c_1^2 \geq 0$. 
Since our manifold has $b_+=1$, 
$H^2(M, \RR)$ is a Lorentzian 
inner-product space, and we now  time-orient this space so that 
$[\omega ]$ is a future-pointing time-like vector. 
Since $c_1^2\geq 0$, it follows that the real Chern class 
$c_1=c_1^\RR$ is a  time-like of null vector in $H^2(M, \RR)$.

Now, because $M$ has $b_+=1$, there are  really two different  Seiberg-Witten 
invariants attached to $M$, even after we have fixed   the  orientation and the  spin$^c$ structure
induced  by $\omega$; namely, the invariant depends on a choice
of one of two {\em chambers} \cite{liu,morgan}. The {\em past  chamber} consists of pairs $(\tilde{g},\eta )$, 
where $\tilde{g}$
is a metric,  $\eta$ is self-dual $2$-form with respect to $\tilde{g}$, and 
$2\pi [c_1]_{\tilde{g}}^+-[\eta_H]$ is a past-pointing  time-like vector;
here $\eta_H$ is the harmonic part of $\eta$ with respect to $\tilde{g}$.
Similarly, the  {\em future chamber} consists of 
pairs $(\tilde{g},\eta)$ for which $2\pi [c_1]_{\tilde{g}}^+-[\eta_H]$ is  future-pointing.
The  Seiberg-Witten 
invariant is essentially   a signed count of gauge-equivalence classes of solutions of the perturbed Seiberg-Witten equations 
\begin{eqnarray*}
\drc_{{\zap A}}\Phi &=&0\\
iF^+_{{\zap A}} +{\textstyle \frac{i}{2}} \Phi \odot \bar{\Phi} &=&  \eta 
\end{eqnarray*}
for any metric $\tilde{g}$ and a generic $\eta$ such that $(\tilde{g},\eta )$ belongs to the given chamber. 
Now Taubes \cite{taubes} has shown that, since $\omega$
is a symplectic form  compatible with a conformal rescaling of  $g$,  the Seiberg-Witten invariant is $\pm 1$ 
for the chamber containing $(g, t \omega )$, when $t\gg 0$; in other words,  the invariant is necessarily non-zero 
for the past  chamber. In particular, there {\em must} be a solution of $(\Phi, {\zap A})$ of the Seiberg-Witten 
equations for any $(\tilde{g}, \eta )$ belonging to the past chamber. On the other hand, 
the perturbed Seiberg-Witten equations imply the $C^0$ estimate
$$|\Phi |^2 \leq \max (2\sqrt{2} |\eta | -s , 0)$$
for any solution with respect to any metric. Since  $M$ admits  a metric $\tilde{g}$ of positive scalar curvature $s$,
it follows  that $(\tilde{g}, \eta )$ cannot belong to the past chamber for any self-dual $2$-form $\eta$ 
of uniformally  small norm. Hence $(\tilde{g}, 0)$ cannot even belong to the {\em closure} of the past chamber, and must therefore
belong to future chamber. Thus,  $[c_1]^+_{\tilde{g}}$ must be a future-pointing time-like vector, and  the non-space-like 
vector $c_1$ must therefore be non-zero and future-pointing. This implies that  $[c_1]^+_g$ is also future pointing, and hence  that 
$c_1\cdot [\omega ] > 0$. 
\end{proof}

This now immediately yields  the following result: 

\begin{thm} \label{window}
Let $M$ be the underlying $4$-manifold of a 
rational or ruled surface with $2\chi + 3 \tau \geq 0$.  
Then any conformal class $[g]$ of symplectic type on $M$ satisfies 
 $$\int_M |W_+|^2 d\mu \geq \frac{4\pi^2}{3} \frac{(c_1\cdot [\omega ])^2}{[\omega ]^2},$$
 with equality iff $[g]$ contains a K\"ahler metric $g$ of constant positive scalar curvature.
\end{thm}
\begin{proof} By Lemma \ref{limpid},  any symplectic conformal class 
 on  $M$ satisfies  $c_1\cdot [\omega ] >0$.
  The 
hypotheses of Proposition \ref{one} are therefore fulfilled by any symplectic 
conformal class $[g]$ on $M$, so  \eqref{proper} holds for every such class. 
Moreover,  because $c_1\cdot [\omega ] > 0$ for any K\"ahler metric on $M$, 
any constant-scalar-curvature metric K\"ahler metric 
on $M$ must have {\em positive} scalar curvature.
Thus,  Proposition \ref{one} predicts that \eqref{proper} is saturated exactly
when the almost-K\"ahler representative $g$ of $[g]$ is 
a K\"ahler metric of constant positive scalar curvature.
\end{proof}

Up to diffeomorphism, the del Pezzo surfaces are exactly the rational or ruled surfaces
with $2\chi + 3\tau > 0$. Theorem \ref{cornerstone}  is therefore just a specialization of Theorem \ref{window}. It is perhaps  worth mentioning, however, 
 that Theorem \ref{window} also applies
to a few  manifolds which do not arise as 
del Pezzo surfaces --- namely,  to     $\CP_2\# 9 \overline{\CP}_2$
and the two  oriented $S^2$-bundles over $T^2$. 

\begin{thm}\label{pane}
Let $M$ be the underlying $4$-manifold of a 
rational or ruled surface with $2\chi + 3 \tau \geq 0$.  
Then any conformal class $[g]$ of symplectic type on $M$ satisfies 
\begin{equation}
\label{simple}
\int_M |W_+|^2 d\mu \geq \frac{4\pi^2}{3} (2\chi + 3\tau ) (M),
\end{equation}
 with equality iff  $[g]$ contains a K\"ahler-Einstein metric 
 $g$ with $\lambda > 0$. In particular, equality never occurs if 
 $(2\chi + 3\tau ) (M) = 0$. 
\end{thm}
\begin{proof} By 
Theorem \ref{window}, every symplectic conformal class $[g]$ on  $M$ 
satisfies 
$$ \int_M |W_+|^2 d\mu  \geq \frac{4\pi^2 }{3} \frac{ (c_1\cdot  [\omega])^2}{[\omega ]^2},$$
with equality only if the conformal class contains a K\"ahler metric $g$
with constant positive scalar curvature $s$. However, the intersection form 
on $H^2 (M, \RR)$ is of Lorentz type, and, for an appropriate time orientation, 
  $[\omega]$ is a  future-pointed  time-like  vector, while $c_1$ is   time-like
  or null future-pointing.  By the reverse Cauchy-Schwarz inequality 
for Minkowski space, we therefore have 
$$c_1\cdot [\omega ] \geq \sqrt{c_1^2} \sqrt{[\omega ]^2},$$
with equality iff $[\omega]$ is a multiple of $c_1$. Theorem \ref{one} therefore implies
that 
 $$\int_M |W_+|^2 d\mu  \geq \frac{4\pi^2 }{3}c_1^2(M),$$
 with equality iff $[\omega ]$ a multiple of $c_1$, and 
  $[g]$ is represented by a K\"ahler metric $g$ of constant positive   scalar curvature.
  In  particular, equality can only happen if $c_1^2 > 0$. 
However, equality would also force the Ricci-form to be harmonic, and belong
 to  the same deRham class as a multiple of the K\"ahler form. 
  Such a metric $g$ would
 necessarily be K\"ahler-Einstein, and, since 
 $c_1\cdot [\omega ] > 0$, this Einstein metric would necessarily have $\lambda > 0$. 
\end{proof}

Restricting to the case when $2\chi + 3\tau > 0$ now yields 
Theorem \ref{keystone}. For related results regarding manifolds with 
$W_+\equiv 0$, see \cite{inyoungthesis}. 

Finally, notice that equality in \eqref{simple} can   never occur   if
 $M=\CP_2\# \overline{\CP}_2$ or $\CP_2\# 2\overline{\CP}_2$, 
since neither of these manifolds admits a K\"ahler-Einstein metric. 
This raises the question of whether a better estimate might be 
possible in these cases. In the next section, we will show  that such an improvement
does at least  hold for  conformal classes on these manifolds which satisfy
an auxiliary symmetry condition.

\section{ Einstein Hermitian Metrics}
\label{hermitian}

As previously mentioned, every del Pezzo surface $(M^4, J)$ admits an Einstein metric $h$
which is Hermitian with respect to the complex structure $J$. In fact \cite{lebhem},
every such metric is conformal to a {\em K\"ahler} metric $g$. Moreover,  this K\"ahler metric 
is automatically extremal. Here  
a K\"ahler metric $g$, with compatible complex structure $J$ and K\"ahler form $\omega$, 
is said to be 
{\em extremal} if the gradient $\nabla s$ of its scalar curvature is a real part of a holomorphic 
vector field. This requirement can be re-expressed in many equivalent ways:
$$\bar{\partial}\nabla^{1,0}s =0 ~\Longleftrightarrow ~\mathscr{L}_{\nabla s}J=0~\Longleftrightarrow ~\mathscr{L}_{J\nabla s}g=0.$$
The last of these conditions  is the most relevant one for present purposes, because
it says that the symplectic vector field on $(M,\omega )$ with Hamiltonian $s$ is actually a Killing field for $g$. 

The concept of an extremal K\"ahler metric was introduced  by Calabi \cite{calabix},
who discovered that,   on a compact complex manifold, the extremality  
condition is equivalent to  the
 Euler-Lagrange equations for $\int s^2 d\mu$, considered  as a functional on all metrics
in a fixed K\"ahler class. 
Calabi's   trailblazing  discoveries  led him to 
 conjecture  \cite{calabix2}  that 
extremal K\"ahler metrics always  {\em minimize} this functional among metrics
in the fixed K\"ahler class, and this conjecture was eventually  proved by Xiuxiong Chen 
\cite{xxel}.   Calabi's  important theorems in the subject include the fact  \cite{calabix2} that, whenever 
$g$ is an extremal K\"ahler metric on a compact complex manifold $(M,J)$, 
the identity  component $\Iso_0(M,g)$ of the isometry group of $g$ is a {\em maximal compact}
connected 
subgroup of the complex automorphism group $\Aut (M,J)$, and so is completely determined, 
up to conjugacy, by the complex structure.

We have already mentioned the fact, first  discovered by Derdzi{\'n}ski \cite{derd}, that 
when a K\"ahler metric
on a $4$-manifold is conformally Einstein, it is also necessarily extremal. This occurs because
$4$-dimensional Einstein metrics are critical points of the Weyl functional, and the 
restriction of $\int |W_+|^2d\mu$ to the space of K\"ahler metrics is $1/24$ times the
Calabi functional  $\int s^2d\mu$. On a compact complex surface $(M,J)$, Derdzi{\'n}ski  also showed that,  if the conformal change
from the Einstein metric $h$ to the K\"ahler metric $g$ is  non-trivial, then the scalar curvature
$s$ of $g$ must be  positive everywhere on $M$, and that, up to a multiplicative  constant, one must have $h=s^{-2}g$. This, together with Calabi's above-mentioned result on isometry groups, allowed the present author to show \cite{lebhem} that this phenomenon
can only occur on  toric del Pezzo surfaces.

Recall that  a compact complex $m$-manifold is said to be {\em toric} if it is of K\"ahler type, 
has non-zero Euler characteristic, and is endowed with an effective  
action of of the $m$-torus $T^m = \RR^m/\ZZ^m$ by biholomorphisms \cite{lebtoric}. 
Any toric
manifold is a smooth projective algebraic variety, and the $T^m$ action always arises
from an algebraic  action of $(\CC^\times)^m$ which  acts freely on an open dense 
orbit, but also has a  finite, non-empty set of fixed points;
in other words, toric manifolds are  just the {\em non-singular  toric
varieties}, in the sense of \cite{fultor}. This point of view makes it particularly apparent
that such manifolds are always birationally equivalent to $\CP_m$. In particular, 
toric manifolds are always simply connected, and have Kodaira dimension $-\infty$.

For our purposes, however, the symplectic perspective will be more important.  
By averaging, any K\"ahler class on a toric  manifold contains K\"ahler metrics which are invariant under the torus action, allowing one to view them as compact symplectic $2m$-manifolds
endowed with Hamiltonian $T^m$-actions. Objects of the latter type are called
{\em Hamiltonian $T$-spaces}, and Guillemin \cite{guiltor} showed that, conversely,  
every  Hamiltonian $T$-space   arises from a toric manifold. 
This makes it seem very natural to put the study of toric K\"ahler metrics into
the context of toric almost-K\"ahler metrics \cite{donfield,dontor,lejmi}.

Let $(M,\omega)$ be a Hamiltonian $T$-space of real dimension $2m$, and let $g$ 
be any almost-K\"ahler metric adapted to $\omega$. Let $\xi_1, \ldots , \xi_m$
be the unit-period vector fields which generate the $T^m$-action, and let
$x^1, \ldots , x^m$ be Hamiltonians for these vector fields:
 $$dx^j = - \xi_j ~\lrcorner ~\omega~.$$
 The map $\vec{x}: M\to \RR^m$ is then called the {\em moment map} of the
 torus action, and its image $P=\vec{x}(M)$, called the {\em moment polytope},
 is always a Delzant polytope in $\RR^m$, meaning that 
 a neighborhood   in $P$ of any vertex 
can be 
transformed into a neighborhood of
   $\vec{0}\in [0,\infty )^m$ 
by an element of $\mathbf{SL}(m, \ZZ) \ltimes \RR^m$;
in particular, every face of $\partial P$  has a unique in-pointing  normal
$1$-form $\nu$ which is an  indivisible element  of the dual integer 
lattice $(\ZZ^m)^*$. 
 Our normalization of the $T^m$-action is  such that the usual $m$-dimensional Euclidean 
  volume measure $d{\zap a}$ on $P$ is exactly the push-forward of the Riemannian
  $2m$-dimensional volume form $d\mu$ on $M$. Even allowing for
  automorphisms of the torus $T^m$, the moment polytope is determined by
  $(M,\omega)$ and the group action up to translations and $SL(m, \ZZ)$ transformations. 
  Thus each face of the boundary $\partial P$ has a natural $(m-1)$-dimensional
  measure $d\lambda$ such that $|\nu \wedge d\lambda |= d{\zap a}$. Indeed,
  $\partial P$ is the image of an anti-canonical divisor in $M$, and $d\lambda$
  is exactly the push-forward of the $(2m-2)$-dimensional Riemannian volume  
  measure on this divisor. 
  
If the
 toric metric $g$  is K\"ahler, Donaldson \cite{dontor,lebtoric} discovered that its scalar curvature
$s$ satisfies the remarkable formula 
$$
\int_P fs~d{\zap a} = 4\pi \int_{\partial P} f ~d\lambda
$$
for any affine-linear function $f: \RR^m$. 
More recently, 
 Lejmi \cite{lejmi} discovered that this beautifully generalizes to become
 \begin{equation}
\label{lej}
\int_P f\varsigma~d{\zap a} = 4\pi \int_{\partial P} f ~d\lambda
\end{equation}
in the almost-K\"ahler case, where $\varsigma$ is the Hermitian scalar curvature, and 
$f$ is again affine linear. The key observation is  that   the matrix of functions 
$H^{jk}: P\to \RR$ defined by $H^{jk}= g(\xi^j, \xi^k)$ satisfies
$$\varsigma = - \sum_{jk} \frac{\partial^2 H^{jk}}{\partial x^j \partial x^k},$$
thereby generalizing  Abreu's celebrated formula for the scalar curvature \cite{abreu}
from the K\"ahler to the almost-K\"ahler setting. After carefully analyzing the boundary behavior
of the $H^{jk}$, 
one then obtains \eqref{lej} by integrating by parts twice, using the fact that 
$\partial^2 f/\partial x^j \partial x^k=0$.

Now let $(M^{2m},g,\omega)$  be a compact toric almost-K\"ahler manifold, and let 
$\mathfrak{t}$ be the Lie algebra of vector fields arising from the given
torus action. The linear space 
$$\mathfrak{K}= \{ f\in C^\infty (M)~|~ \exists \xi \in \mathfrak{t} ~s.t.~  
df = - \xi ~\lrcorner ~\omega\},$$
of  Hamiltonians for  Killing fields belonging to $\mathfrak{t}$
then has dimension $m+1$, and is exactly the pull-back of the affine-linear
functions on $\RR^m$ via the moment map $\vec{x}$. If $\thor : L^2 (M) \to \mathfrak{K}$
is the $L^2$-orthogonal projection, then, because $d{\zap a}$ is exactly the
push-forward of $d\mu$, \eqref{lej} asserts that $\thor (\varsigma)$ is completely
determined by the moment polytope $P$. In particular, the affine-linear function 
we obtain in this way is exactly the same as it would be in the special case in which 
our chosen metric is K\"ahler! 

In particular, this construction allows us to extend the usual definition 
of the {\em Futaki invariant} $\mathfrak{F}: \mathfrak{t}\to \RR$ to apply to any 
toric almost-K\"ahler manifold. (However, the usual Futaki invariant \cite{fuma,fuma0}
is conceived of as acting on arbitrary holomorphic vector fields; in the present context, 
it will  merely be defined on Killing fields.) Namely, one first identifies 
$\mathfrak{t}$ with the subspace $\mathfrak{K}_0\subset \mathfrak{K}$
of functions with integral $0$, and then taking the $L^2$ inner product of elements
of $\mathfrak{K}_0$ with $-\varsigma$ to obtain an element of $\mathfrak{t}^*$.
If  $\bar{\vec{x}}$ is the barycenter of $P$ with respect to the standard Euclidean
measure $d{\zap a}$, then $\mathfrak{t}=\mathfrak{k}_0$ can be identified
with the affine linear functions which vanish at $\bar{\vec{x}}$. If $\langle \vec{x}\rangle$
instead denotes the barycenter of $\partial P$ with respect to $d\lambda$,
then 
$$f\in \mathfrak{K}_0 \Longrightarrow \mathfrak{F}(f)
= -4\pi |\partial P|~f ( \langle \vec{x}\rangle ),$$
where 
$$
|\partial P| =\int_{\partial P} d\lambda ={\textstyle \frac{1}{(m-1)!}}~ c_1\cdot [\omega]^{m-1}
$$
is the $\lambda$ measure of $|\partial P|$. 
We may thus make the identification 
$$\mathfrak{F}= - 4\pi |\partial P| \vec{\mathfrak{D}}\in \mathfrak{K}^*$$
where 
$$
\vec{\mathfrak{D}}= \langle \vec{x}\rangle - \bar{\vec{x}}
$$
is the  displacement vector representing the separation between the 
barycenter of the interior and boundary of $P$.

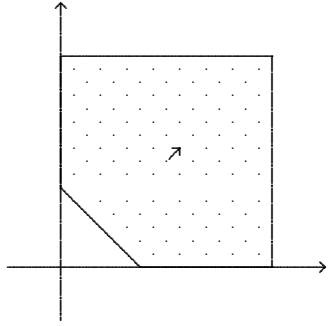
\begin{center}
\mbox{
\beginpicture
\setplotarea x from 0 to 120, y from 0 to 130 
\arrow <2pt> [1,2] from 0 20  to 120 20
\arrow <2pt> [1,2] from 20 0 to 20 120 
\arrow <2pt> [1,2] from 61 61 to 65 65  
\hshade 20 50  100 50  20 100  100 20 100  /  
{\setlinear
\plot 20 50 20 100 /
\plot 50 20  100  20 / 
\plot 20 50  50 20 /
\plot 20 100 100 100 /
\plot 100 100 100 20 / 
}
\endpicture
}
\end{center}

The  {\em moment-of-inertia} matrix  matrix $\Pi$ defined by 
$$\Pi_{jk} = \int_P (x_j-\bar{x}_j)(x_k-\bar{x}_k)d{\zap a}$$
now represents the $L^2$ inner product on $\mathfrak{K}_0$, and 
$\Pi^{-1}\vec{\mathfrak{D}}$ represents $\thor (\varsigma) - \bar{\varsigma}$ in terms of the basis
$\{ x_j-\bar{x}_j\}$ for $\mathfrak{K}_0$; here
$$
\bar{\varsigma}=2\pi m  \frac{c_1\cdot [\omega]^{m-1}}{[\omega]^m}= 4\pi \frac{|\partial P|}{|P|}
$$
is the average value of the Hermitian scalar curvature $\varsigma$,  the last expression for which involves 
 the $m$-dimensional Euclidean volume
$$|P|  = \int_Pd{\zap a} = \int_M d\mu ={\textstyle \frac{1}{m!}}[\omega]^m$$
  of the moment polytope $P$. 
Thus
$$
\int_M[\thor (\varsigma ) -\bar{\varsigma}]^2d\mu = 16\pi^2|\partial P|^2 \vec{\mathfrak{D}}\cdot \Pi^{-1} \vec{\mathfrak{D}}$$
and
$$
\int_M [\thor (\varsigma )]^2 d\mu = 16\pi^2  |\partial P|^2\left( \frac{1}{|P|}+  \vec{\mathfrak{D}}\cdot \Pi^{-1} \vec{\mathfrak{D}}
\right) .
$$

\begin{defn}
Let $(M^{2m},\omega)$ be a Hamiltonian $T$-space. We then define the {\em virtual
action} associated with the given torus action to be 
\begin{equation}
\label{version1}
{\mathcal A}([\omega ]) = \frac{|\partial P|^2}{2}  \left( \frac{1}{|P|}+  \vec{\mathfrak{D}}\cdot \Pi^{-1} \vec{\mathfrak{D}}
\right) ,
\end{equation}
where $\vec{\mathfrak{D}}$ is the vector joining the barycenter of $P$ to the barycenter
of $\partial P$, $\Pi$ is the moment-of-inertia matrix of $P$,and where $|P|$ and $|\partial P|$ 
respectively denote the ${\zap a}$-measure of the moment polytope and $\lambda$-measure of  its boundary.\end{defn}

The above discussion then immediately gives us the following result:

\begin{prop} Let $(M,g,\omega)$ be any compact toric almost-K\"ahler manifold. Then
the Hermitian scalar curvature $\varsigma$ satisfies 
$$
\frac{1}{32\pi^2}\int_M \varsigma^2d\mu \geq  \mathcal{A} ([\omega ]),
$$
with equality iff $J\nabla \varsigma$ is a Killing field of $g$. 
\end{prop}
\begin{proof}
Since $\thor$ is the  orthogonal projection $L^2(M)\to \mathfrak{K}$, the Pythagorean theorem
tells us that 
$$\int_M\varsigma^2 d\mu \geq \int_M [ \thor (\varsigma )]^2 d\mu~,$$
with equality if $\varsigma \in \mathfrak{K}$. Since we have just seen that 
$$
 \int_M [ \thor (\varsigma )]^2 d\mu = 32\pi^2 \mathcal{A}([\omega ])~,
$$
this establishes the desired inequality. 

It only remains to observe that if $\varsigma$ is the Hamiltonian of a Killing field $\xi$, then 
it must actually belong to $\mathfrak{K}$. This is true because $\varsigma$ Poisson commutes
with $x_1, \ldots , x_m$, so the 
 torus in $\Iso (M,g)$ given by $\overline{\{ \exp (t\xi)~|~ t\in \RR \}}$ would necessarily, by continuity,  consist of Hamiltonian symplectomorphisms whose 
Hamiltonians  would all Poisson commute with $x_1, \ldots , x_m$, too. The 
Hamiltonians $f_1, \ldots , f_k$ of the periodic generators of this torus  would therefore be functions of $(x_1, \ldots , x_m)$, and the periodicity of the corresponding vector fields
 would force the gradients of these functions on $\RR^m$  to necessarily belong   to the integer lattice at every point of $P$. The gradients $\nabla f_1, \ldots , \nabla f_k$ would therefore be constant, 
 and $f_1, \ldots , f_k$ would thus be  affine linear functions on 
$\RR^m$. It follows that $f_1, \ldots , f_k\in \mathfrak{K}$. But $\varsigma$ is, 
mod constants,  a linear combination of 
$f_1, \ldots , f_k$,   so it then follows  that $\varsigma$ must  belong to  $\mathfrak{K}$, too. 
\end{proof}

We now specialize to the case of real dimension $4$, where the virtual action
\begin{equation}
\label{version2}
\mathcal{A}([\omega ])= \frac{(c_1\cdot [\omega ])^2}{[\omega ]^2} + \frac{1}{32\pi^2} \|\mathfrak{F}([\omega ])\|^2
\end{equation}
has already been studied elsewhere \cite{lebuniq,lebtoric,lebhem10} for toric del Pezzo surfaces. 
Notice that this $m=2$ case  enjoys the special property that $\mathcal A$ becomes {\em scale invariant}; 
in particular, $\mathcal A$ is unchanged in this setting if we rescale moment polygon
by dilating the picture. This primarily reflects the fact that $L^2$ norms of 
curvature are scale invariant in dimension $4$; however, it also has the added benefit
that we would have obtained precisely the same formula virtual action if our conventions had assigned
period $2\pi$ (or any other fixed number) to each of the generating vector fields of
our torus action, instead of taking them to have period $1$. Note, however, that 
this scale invariance does not persist in higher dimensions, so that another choice
of conventions would have resulted in a formula for $\mathcal A$ which involved factors 
depending on the dimension.

Now let $M$ be a simply connected 
$4$-manifold which has been  equipped with some fixed $T^2$-action, and suppose that 
there is at least one symplectic structure $\omega_0$ which is $T^2$-invariant; in particular, 
this implies that the intersection form on $H^2(M,\RR)$ is of Lorentz type. Time-orient 
$H^2(M, \RR)$ so that $[\omega_0]$ is future-pointing, and let 
 ${\zap K}\subset H^2(M, \RR)$ denote the set of future-pointing of cohomology classes 
 on $M$ which are represented by $T^2$-invariant symplectic forms on $M$.
 We will call $\mathscr K$ the {\em symplectic cone} of our toric manifold;
 it is automatically open, and is invariant under the action of the positive 
 reals numbers $\RR^+$ by scalar multiplication. 
 We will call the quotient $\check{{\zap K}}$ the {\em reduced symplectic cone}
 of the toric manifold. Since the virtual action $\mathcal A$ is invariant under rescaling, 
 we will find it convenient to  view it as a function 
 $${\mathcal A}: \check{{\zap K}}\to \RR.$$

 As is shown by Guillemin \cite[Appendix 2]{guiltor}, these are just the the classes which arise from 
 K\"ahler metrics for some $T^2$-invariant complex structure on $M$. For
 toric del Pezzo surfaces, the complex structure with $c_1 > 0$ is actually {\em unique},
 so that  ${\zap K}$ can simply be identified with the {\em K\"ahler cone} of the corresponding
 del Pezzo surface. In particular, 
  ${\zap K}$ is then an open convex cone in $H^2(M, \RR)$, and the reduced cone 
  $\check{{\zap K}}$ can be viewed as  a convex open set in a Euclidean space
  of dimension $b_-(M)=b_2(M)-1$.

 \begin{prop} 
\label{behold}
Let $M$ be a toric del Pezzo surface, equipped with the associated 
$T^2$-action. Then the virtual action $$\mathcal{A}: \check{\zap K}\to \RR $$
has exactly one critical point.
This unique critical point occurs at the  absolute minimum of $\mathcal{A}$,
and is   the K\"ahler class of the unique conformally Einstein, 
K\"ahler metric on $M$.
\end{prop}
\begin{proof}
There are exactly five toric del Pezzo surfaces. 

Two of these, $\CP_2$ and 
$\CP_1\times \CP_1$, have semi-simple automorphism groups, and hence
vanishing Futaki invariant (because the Futaki invariant must be invariant under
the adjoint action). In these cases, we therefore just have 
${\mathcal A}([\omega ])  = (c_1\cdot [\omega)^2/[\omega ]^2$, and the unique
critical point therefore occurs at the absolute minimum $c_1$. This is of course the K\"ahler class of
the obvious K\"ahler-Einstein metric on either of these two manifolds. 

The three-point blow-up $\CP_2\# 3\overline{\CP}_2$ of the projective plane
also admits a K\"ahler-Einstein metric,  first discovered by Siu \cite{s},
and the uniqueness of this metric (up to automorphisms and rescalings)
follows from the Bando-Mabuchi theorem \cite{bama}. Because 
${\mathcal A}([\omega ])  \geq (c_1\cdot [\omega ])^2/[\omega ]^2$, with equality when $[\omega ] = c_1$, 
the reverse Cauchy-Schwarz inequality for Minkowski space predicts that
 the unique minimum of ${\mathcal A}$ must occur
at the K\"ahler class $c_1$ of this K\"ahler-Einstein metric. However, the fact that there
are no other critical points of $\mathcal A$ is a more delicate issue, and was proved in 
\cite[Proposition 4]{lebuniq}.

The existence of an Einstein,  conformally K\"ahler metric on the 
two-point blow-up $\CP_2\# 2\overline{\CP}_2$ of the projective plane was first
proved in \cite{chenlebweb}; see  \cite{lebhem10} for a somewhat different proof. 
It was them proved in  \cite[Proposition 3]{lebuniq} that 
 the corresponding
critical point of $\mathcal A$ is unique, and is in fact a global minimum. 
  By \cite{lebhem}, this reduces the uniqueness
 problem for the Einstein metric to the uniqueness (modulo automorphisms)
 of extremal Ka\"ahler metrics in a fixed K\"ahler class, which had previously been proved 
 in \cite{xxgang2}.

The relevant Einstein metric on the one-point blow-up
$\CP_2\# \overline{\CP}_2$ of the projective plane was constructed explicitly 
by Page \cite{page}, although it was Derdzi{\'n}ski \cite{derd} who later discovered
that it is in fact conformally K\"ahler. Uniqueness of this metric has usually
been proved via symmetry arguments \cite{lebhem}, but here we will need to 
reprove this by examining the virtual action $\mathcal A$. 
 By rescaling, we can 
arrange that the moment polygon takes the form
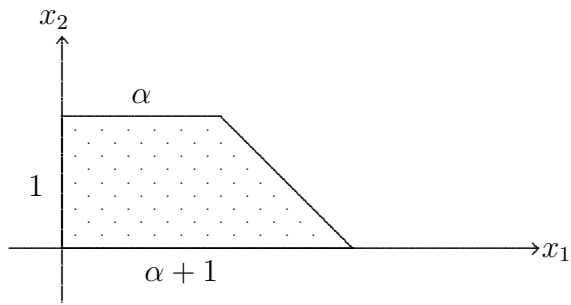
\begin{center}
\mbox{
\beginpicture
\setplotarea x from 0 to 180, y from 0 to 110 
\arrow <2pt> [1,2] from 0 20  to 200 20
\arrow <2pt> [1,2] from 20 0 to 20 100 
\put {$x_1$} [B1] at 207 17  
\put {$x_2$} [B1] at  17 105  
\put {$1$} [B1] at 10 40  
\put {$\alpha$} [B1] at 50 75 
\put {$\alpha+1$} [B1] at 65 8  
\hshade 20 20  130 70  20 80    /  
{\setlinear
\plot 20 20 20 70 /
\plot 20 20  130 20 / 
\plot 20 70 80 70 /
\plot 80 70 130 20 / 
}
\endpicture
}
\end{center}
and the virtual action is then given \cite{lebtoric} by 
$$
{\mathcal A}([\omega ]) =  \frac{12\alpha^3+42\alpha^2+48\alpha+
9}{6\alpha^2+6\alpha + 1}~.
$$
This must have a minimal value for $\alpha \in (0,\infty )$ because it behaves like the decreasing function $9-6\alpha$ when $\alpha$ is small, and like the increasing function $2\alpha$
when $\alpha$ is large. Moreover, this critical point is unique, because 
$$\frac{d^2{\mathcal A}}{d\alpha^2}  = 48\cdot \frac{24\alpha^3+18\alpha^2+6\alpha +1}{(6\alpha^2+6\alpha+1)^3}$$
is positive when $\alpha > 0$. 
We may thus   once again appeal to 
\cite{xxgang2} to obtain a heavy-handed proof of the 
uniqueness (up to rescalings and automorphisms) of the Einstein, Hermitian metric
on this space ---  although this uniqueness statement can actually be proved in more elementary ways \cite{lebhem,lebuniq}.
\end{proof}

To prove our next main result, we will need the following fact:

\begin{lem}\label{positive}
Let $M$ be a toric del Pezzo surface, equipped with the associated $T^2$-action. 
Let $g$ be any almost-K\"ahler metric on $M$ which is $T^2$-invariant. Then 
$\thor (\varsigma ) \geq 0$  at every point of $M$.
\end{lem}
\begin{proof} As we have seen, $\thor (\sigma )$ only depends on the moment polygon,
and indeed is explicitly given by the affine-linear function 
 $$\thor (\varsigma ) =  4\pi \left( \frac{|\partial P|}{|P|}+(\vec{x}-\bar{\vec{x}})\cdot \Pi^{-1}\vec{\mathfrak{D}}\right)$$
on the moment polygon $P$. 
 For the three-point blow-up  $\CP_2\# 3\overline{\CP}_2$ of the complex
projective plane,   \cite[Lemma B.2]{lebhem10} proves that $\thor (\varsigma )> 0$ on $P$ 
by showing that $\thor (\varsigma )$ is positive at every vertex of $P$. 
(While the statement of that Lemma  merely  emphasizes consequences for 
the behavior of putative extremal  K\"ahler metrics, the proof precisely show that 
$\thor (\varsigma ) > 0$.)
On the other hand, the above  explicit formula for $\thor( \varsigma )$ is clearly continuous under
degeneration of the polygon by letting  a side shrink  to length  zero. Since every other 
toric del Pezzo surface is a blow-down of $\CP_2\# 3\overline{\CP}_2$, it therefore follows
that $\thor (\varsigma ) \geq 0$ everywhere. 
\end{proof}

\begin{rmk} It can in fact be shown that $\thor (\varsigma ) > 0$ everywhere on 
any toric del Pezzo surface; cf.  \cite[Lemma A.2]{lebhem10}. However,  this more
delicate fact will not be needed for our present purposes. 
\end{rmk}

We are now in position to prove the following:

\begin{thm}\label{two}
Let $M$ be the underlying $4$-manifold of a  toric del Pezzo surface, equipped with 
the associated  $T^2$ action. Then any toric 
conformal class  $[g]$  of symplectic type on $M$ satisfies
$$ \int_M |W_+|^2 d\mu  \geq \frac{4\pi^2 }{3} {\mathcal A}([\omega ])$$
where the virtual action ${\mathcal  A}$ is given by either \eqref{version1} or \eqref{version2}.
Moreover, equality holds iff $[g]$ contains an extremal K\"ahler metric $g$. 
\end{thm}
\begin{proof}
Let $[g]$ be a  symplectic  conformal class which is invariant under the fixed $T^2$-action, 
with  harmonic $2$-form $\omega\neq 0$. The conformal invariance of harmonic $2$-forms then implies that $\omega$ is also $T^2$-invariant. Let $g\in [g]$ be the unique metric in the conformal class
such that $|\omega |\equiv \sqrt{2}$. It then follows that $g$ is a  $T^2$-invariant
almost-K\"ahler metric compatible with the symplectic form $\omega$. 
Since $\int |W_+|^2d\mu$ is conformally invariant, we will once again prove
the desired statement about the conformal class $[g]$ by evaluating the integral with respect to $g$. 

Let us   now set
 $$F = \thor (\varsigma ),$$
 where 
$$\varsigma = \frac{s+s^*}{2}$$ is  again the Hermitian scalar curvature
of $(M,g,\omega)$,  and where 
 $$\thor : C^\infty (M) \to \mathfrak{K}$$ 
is  again the $L^2$ projection to the 
the subspace 
$$\mathfrak{K}= \{ f\in C^\infty (M)~|~ \exists \xi \in \mathfrak{t} 
~s.t.~  df = - \xi ~\lrcorner ~\omega\}$$
of Hamiltonians for  $T^2$-invariant Killing fields.
Notice  that Lemma \ref{positive} then tells us that $F\geq 0$
everywhere on $M$. Moreover, since 
$$
\bar{F}= \bar{\varsigma} = 4\pi \frac{|\partial P|}{|P|}> 0,
$$
and since $F$ is the pull-back of an affine-linear function on $P$, we in fact have $F> 0$ on an open dense set of $M$.

 Now  Lemma \ref{key} tells us that  
 $$2\sqrt{6} |W_+| \geq \varsigma$$
 at every point of $M$, with equality exactly at those points where $\nabla \omega =0$ and $s \geq 0$. Multiplying by $F\geq 0$, we therefore have 
 $$ 2\sqrt{6} ~F |W_+| \geq F \varsigma ~,$$
and, since $F> 0$ on a dense set,   equality occurs 
everywhere  iff $g$ is K\"ahler, with $s \geq 0$. 
 Integrating over $M$, we therefore have
 \begin{eqnarray*}
2\sqrt{6} \int_M F |W_+| d\mu &\geq& \int_M F \varsigma d\mu\\
 &=& \int_M  \thor (\varsigma ) \varsigma  ~d\mu\\
  &=& \int_M \left[\thor (\varsigma  )\right]^2 d\mu\\
 &=& \int_M F^2 ~d\mu,
\end{eqnarray*}
 since $\thor$ is an orthogonal projection with respect to the $L^2$ inner product; moreover, equality holds throughout 
  iff $g$ is an {\em extremal} K\"ahler metric. On the other hand, the Cauchy-Schwarz
  inequality tells us that 
$$\left(\int_M F^2 d\mu\right)^{1/2}\left(\int_M |W_+|^2 d\mu\right)^{1/2}\geq \int_M F |W_+| d\mu ,$$ 
so that   
$$2\sqrt{6}\left(\int_M F^2 d\mu\right)^{1/2}\left(\int_M |W_+|^2 d\mu\right)^{1/2}\geq \int_M F^2 d\mu~.$$ 
It follows that 
$$\int_M |W_+|^2 d\mu \geq\left( \frac{1}{2\sqrt{6}}\right)^2\int_M F^2 d\mu =\frac{4\pi^2}{3} \mathcal{A}([\omega ]) ,$$ 
with equality   iff $g$ is an extremal K\"ahler metric.
\end{proof}

In particular,  Theorem \ref{two} tells us that 
any toric symplectic class on a del Pezzo surface satisfies 
$$
\int_M |W_+|^2 d\mu \geq \frac{4\pi^2}{3} \min_{[\omega ]\in {\zap K}} \mathcal{A}([\omega ]),
$$
with equality iff $[g]$ contains an extremal K\"ahler metric $g$ 
which minimizes $\mathcal{A}$. But   Proposition \ref{behold} 
asserts that any such minimizing $g$ is conformally related to an Einstein metric,
that any Einstein Hermitian metric arises in this way, and that,
up to automorphisms and rescalings, exactly one such Einstein 
metric exists on each toric del Pezzo surface. This proves Theorem \ref{capstone}.

\section{Einstein Metrics and the Period Map}
\label{period} 

  Given a smooth Riemannian metric $g$ on a compact oriented $4$-manifold $M$ with 
  $b_+=1$, the space of self-dual $2$-forms gives us a time-like $1$-dimension subspace
  $\RR[ \omega ]$ in $H^2(M, \RR)$. The map 
  $g\longmapsto \RR[\omega ]$ 
  from the space of metrics to 
  an open ball in the real projective space $\mathbb{P} (H^2(M, \RR))$ is sometimes
  called the {\em period map} of $M$. If $g$ happens to be a K\"ahler-Einstein metric
  with $\lambda > 0$, we would then of course know that $\RR [\omega ] = \RR c_1$.
 But even if $g$ is  allowed to be an {\em arbitrary} Einstein metric 
 on a del Pezzo surface, then, provided that 
  $[g]$ is assumed to be of symplectic type, we will  show that the lines 
  $\RR [\omega ]$ and $\RR c_1(M,\omega)$ cannot be too far apart.

 \begin{thm} \label{yellowstone}
 If  a conformal class $[g]$  on a 
 del Pezzo surface  $M$ is of symplectic type and satisfies 
 $$\frac{(c_1\cdot [\omega ])^2}{[\omega]^2} \geq \frac{3}{2} c_1^2 (M)$$
 then $[g]$ is not the conformal class of   an Einstein metric. 
 \end{thm}

  \begin{proof}
  The Gauss-Bonnet-type formula 
  $$(2\chi + 3\tau)(M) = \frac{1}{4\pi^2}
  \int_M \left( \frac{s^2}{24}+2|W_+|^2 -\frac{|\mathring{r}|^2}{2}
  \right) d\mu $$
  implies that any Einstein metric on a compact  almost-complex $4$-manifold must satisfy 
  \begin{equation}
\label{taut}
   \int_M |W_+|^2 d\mu \leq 2\pi^2 c_1^2 (M) ,
\end{equation}
  with equality only if the Einstein metric is Ricci-flat. 
  Combining this with the inequality of Theorem \ref{window} thus yields
  \begin{equation*}
\label{patrol}
 \frac{(c_1\cdot [\omega ])^2}{[\omega ]^2} < \frac{3}{2} c_1^2~,
\end{equation*}
  where the inequality is necessarily strict because a Ricci-flat metric 
 cannot be conformally related to a metric of  positive scalar curvature. The result therefore follows by contraposition. 
  \end{proof}
  
  \begin{xpl}
  Let $[g]$ be a conformal class of symplectic type on $M=S^2\times S^2$, with 
  self-dual harmonic $2$-form $\omega$. By moving $[g]$ by a diffeomorphism
  and rescaling $\omega$ if necessary, we may assume \cite{lalmcd}  that $[\omega ]$ 
  is Poincar\'e dual to $F_1 + t F_2$  and  that $c_1$ is 
 is Poincar\'e dual to  $2F_1 + 2 F_2$,
 where $F_1 = S^2 \times \{ pt\}$,
  $F_2 = \{ pt\} \times S^2$, and $t\geq 1$.  Theorem \ref{yellowstone} then tells us that 
  $[g]$ cannot be the conformal class of an Einstein metric if 
  $$
  \frac{(2t+2)^2}{2t}\geq 12 ,
  $$
  which  happens whenever $t\geq 2 + \sqrt{3}$. 
  \end{xpl}

  In the toric setting,  Theorem \ref{yellowstone} can be improved as follows:
  
 \begin{thm} \label{moonstone}
 Let $M$ be a toric complex surface, and let $[g]$ be a $T^2$-invariant conformal class on $M$ 
 of symplectic type. If $[g]$  also satisfies 
 $$\frac{(c_1\cdot [\omega ])^2}{[\omega]^2} +{\frac{1}{32\pi^2}}\| \mathcal{F} ([\omega ])\|^2 \geq \frac{3}{2} c_1^2 (M)$$
 then $[g]$ is not the conformal class of   an Einstein metric. Here $\|\mathcal{F}\|^2$ once again 
 denotes the norm-square of the Futaki invariant.
 \end{thm}
 
\begin{proof} Any del Pezzo surface is simply connected. 
Since $g$ is an  Einstein metric with {\em toric}  conformal class  $[g]$, 
it  therefore cannot be Ricci-flat, 
 because  a Bochner-type  argument due to Lichnerowicz 
 \cite{lichconf} shows that
 a compact Ricci-flat manifold with $b_1=0$ cannot admit  conformal Killing fields. 
Hence   strict inequality holds in  \eqref{taut}. Combining 
this inequality with  Theorem \ref{two} now yields 
\begin{equation}
\label{control}
\frac{(c_1\cdot [\omega ])^2}{[\omega]^2} +{\frac{1}{32\pi^2}}\| \mathcal{F} ([\omega ])\|^2 < \frac{3}{2} c_1^2 (M),
\end{equation}
 and the result   follows by contraposition. 
\end{proof}

  Note that the inequality \eqref{control}   plays  an important role in the existence 
  theory of Einstein, 
  conformally K\"ahler metrics, where the 
   region it defines  is the called {\em controlled cone} \cite{chenlebweb,lebhem10}.

\section{Problems and Prospects}
\label{speculations}

  Theorems \ref{keystone} and \ref{capstone}, together with Gursky's Theorem, 
  provide  interesting evidence in favor of the following:
  
  \begin{conj} Let $M$ be a smooth compact Einstein metric $g$ with positive Einstein constant.
  Suppose, moreover, that $g$ is Hermitian with respect to some integrable complex structure
  $J$ on $M$. Then the conformal class $[g]$ is an absolute minimizer of the Weyl functional 
  $\mathscr{W}$. Moreover, for the given $M$, every absolute minimizer arises in this manner. 
  \end{conj}

Of course, the techniques developed in this article are by no means 
sufficient to prove such a result. 
For a generic conformal class on $M$, the corresponding harmonic self-dual $2$-form $\omega$ 
will vanish along a union of circles \cite{kirbyperiod,perutz,taubescirc}, 
resulting in an incomplete almost-K\"ahler structure
on the complement, and determining a first Chern class which changes as circles are 
created or annihilated. While one might hope to systematically extend the techniques 
developed here to cope with these difficulties, it remains to be seen whether such a direct 
assault on the problem could actually work. An interesting related problem would
 be to  determine whether 
such circles can be introduced without forcing the Yamabe invariant to  become negative. 

Extending the estimate of  Theorem \ref{capstone} to non-toric conformal classes would be 
 desirable, but is technically extremely daunting. In the K\"ahler case, the corresponding
statement was proved by X.X. Chen \cite{xxel}, but the ideas involved are unlikely to 
generalize in any direct way to the 
almost-K\"ahler setting. 
A small but interesting first step, however, might be to  simply prove that 
Theorem \ref{capstone} still holds if the toric structure is no longer fixed, but is allowed to 
range over all possible toric structures on the fixed $4$-manifold. Here the 
main difficulty is that our proof of 
 Theorem  \ref{two} depends on the fact that $\thor (\varsigma )$ is everywhere non-negative, 
 and this will simply not be true for most toric structures. 
 
 On the other hand, it
 is  worth noting  that 
any toric conformal class automatically satisfies $c_1\cdot [\omega ] > 0$,
even if $c_1^2 < 0$, simply because $c_1$ is  Poincar\'e dual to the bracelet of 
symplectic $2$-spheres represented by the boundary of the moment polygon. For
this reason, 
Proposition \ref{one} is broadly applicable in the toric setting, and might be a source of
useful related  results.

\appendix
\section{Appendix: The Yamabe Invariant}
\label{appa}

In this appendix, we contrast  Gursky's Theorem \cite{gursky} with Theorem \ref{keystone}
by showing that ``most'' symplectic conformal classes on any $4$-manifold have negative
Yamabe constant. For related results, see \cite{jongsuak,jongsuak2}.

\begin{prop} \label{neg}
Let $(M^4 ,\omega)$ be any compact symplectic $4$-manifold. 
Then there are sequences of conformal classes $[g_k]$ on $M$ which 
are compatible with $\omega$, and have Yamabe constants $Y([g_k]) \to -\infty$. 
Moreover, among all\ $\omega$-compatible conformal classes, those with negative
Yamabe constant are dense in the $C^0$ topology. 
\end{prop}
\begin{proof}
The Yamabe constant of a conformal class $[g]$ on a compact $4$-manifold $M$ 
is given by 
$$Y([g]) = \inf_{g^\prime \in [g]} \frac{\int_M s_{g^\prime}d\mu_{g^\prime}}{\sqrt{\int_M ~d\mu_{g^\prime}}} ,$$
so it suffices to produce a sequence of almost-K\"ahler metrics $g_k$  adapted
to $\omega$ such that $\int s_{g_k}~d\mu_{g_k}\to -\infty$; here we are using the fact that 
all such metrics have the same volume form $d\mu=\omega^2/2$, and hence the same total volume. On the other hand, for any almost-K\"ahler metric $g$,
$$
4\pi c_1\cdot [\omega ] = \int_M \frac{s+s^*}{2}d\mu =  \int_M \left(s+  \frac{|\nabla \omega|^2}{2}\right) d\mu
$$
so it suffices to show that we can choose $g_k$ adapted to $\omega$ so that 
$$\| \nabla \omega \|_{L^2}\to +\infty.$$

Of course, a choice of almost-K\"ahler metric $g$ compatible with $\omega$ is 
equivalent to the choice of an almost-complex structure $J$ such that 
$\omega$ is $J$-invariant; the metric is then given by 
$$g = \omega (\cdot , J \cdot ), $$
and, with respect to this metric, $J$ then  just becomes $\omega$ with an index raised.
The Nijenhuis tensor $N_J$, given by 
$$N_J (X,Y)=( \nabla_X J)(Y)-(\nabla_Y J)(X)-( \nabla_{JX} J)(JY)+(\nabla_{JY} J)(JX)$$
is 
 then just  four times the projection of $\nabla J$ to $(\Lambda^{0,2}_J\oplus \Lambda^{2,0}_J)\otimes TM$, and represents the O'Neill tensor of $T^{0,1}_J$ in that 
 $$[X+iJX,Y+iJY]^{1,0}=J N_J(X,Y)+iN_J(X,Y).$$
It thus suffices to produce a sequence $J_k$ of almost-complex structures with  
$J_k^*\omega=\omega$ which satisfy 
$$\|N_{J_k}\|_{L^2}\to \infty ,$$
where the norms are to be computed with respect to the associated sequence of 
metrics 
$$g_k = \omega (\cdot , J_k \cdot ).$$

This can be done via an entirely local construction. First,  choose an arbitrary background
almost-K\"ahler metric $g$ adapted to $\omega$, which amounts to  choosing 
a background $\omega$-compatible almost-complex structure $J$. Now  take a Darboux
chart $(x,y,u,v)$ near an arbitrary point $x\in M$, so that 
\begin{equation}\label{darboux}
\omega = dx\wedge dy + du\wedge dv
\end{equation}
and such that $J$ coincides with the standard Euclidean almost-complex structure 
at the origin, which  represents $x$  
in these coordinates. Next,  by freezing the almost-complex structure near the origin, introduce a perturbed back-ground almost-K\"ahler metric 
$g_0=g_{0,\varepsilon}$ which  agrees with the standard Euclidean metric on, say, the
coordinate ball of radius $3\varepsilon$ about the origin, but agrees with $g$ outside
the ball of radius $4\varepsilon$; and notice that we can do this so that $g_{0,\varepsilon}\to 
g$ in the $C^0$ topology as $\varepsilon \to 0$. Next, for any fixed $\varepsilon$, notice that we  can  construct a
new almost-K\"ahler metric $g_f$ associated with an almost complex structure
given by 
$$
J_f= e^{2f} dx \otimes \frac{\partial}{\partial y}- e^{-2f} dy \otimes \frac{\partial}{\partial x}
+ du \otimes \frac{\partial}{\partial v}-  dv \otimes \frac{\partial}{\partial u}
$$
on the $(3\varepsilon)$-ball, where $f$ is a smooth function supported in the ball 
of radius $2\varepsilon$; we then extend $J$ to all of $M$ by taking it to coincide with 
the almost-complex structure $J_0$ of  $g_0$ outside the 
$(2\varepsilon)$-ball. The corresponding metric $g_f$ is given by 
\begin{equation}
\label{ansatz}
g_f= e^{2f}dx^2 + e^{-2f}dy^2 + du^2 + dv^2
\end{equation}
on the region in question, and 
$$(e^fdx+ie^{-f}dy) \left(\left[e^{-f} \frac{\partial}{\partial x} + i e^f  \frac{\partial}{\partial y}, \frac{\partial}{\partial u} +i \frac{\partial}{\partial v}\right]\right)= 2\left(\frac{\partial}{\partial u} +i \frac{\partial}{\partial v}\right) f$$
is a component of $JN_J+iN_J$ in an orthonormal frame. Now let $g_k=g_{k,\varepsilon}$ be  the sequence 
of almost-K\"ahler metrics $g_{f_k}$  associated with the sequence of functions $f_k =  \frac{1}{k} \sin (\frac{2\pi k^2}{\varepsilon} v) \phi$, where the cut-off function $\phi : M\to [0,1]$
 is supported in the $(2\varepsilon)$-ball and 
  $\equiv 1$ on 
the $\varepsilon$-ball. The Nijenhuis tensors
of the corresponding almost-complex structures then satisfy 
$$\| N_{J_k}\|_{L^2}^2 > \int_{[-\frac{\varepsilon}{2}, \frac{\varepsilon}{2}]^4}
 |\frac{\partial f}{\partial v}|^2~ \frac{\omega^2}{2} 
= 2\pi^2 k^2 \varepsilon^4 \to + \infty
$$
with respect to the associated metrics. 

In particular, for fixed $\varepsilon$, the Yamabe constant of
$[g_k] = [g_{k,\varepsilon}]$ is negative for all large $k$. However, 
 the above choice of $f_k$ converges to $0$ in $C^0$,
so that $g_{k,\varepsilon}\to g_{0,\varepsilon}$ in the $C^0$ topology as $k\to \infty$. On the other hand,
we also have $g_{0,\varepsilon}\to g$ in the $C^0$ topology  as $\varepsilon \to 0$. 
By choosing a suitably large $k(j)$ for each $\varepsilon = 2^{-j}$,
we therefore obtain a sequence of $\omega$-compatible metrics converging to the given 
$\omega$-compatible metric  $g$
in $C^0$, even though  
their  conformal classes   have negative Yamabe constants. 
\end{proof}

In fact, Jongsu Kim \cite{jongsuak,jongsuak2} has proved much stronger and more difficult
results in this direction. Indeed,  he shows that one can always construct $\omega$-compatible 
almost-K\"ahler metrics $g$ for which the scalar curvature is everywhere negative, 
even without conformal rescaling. 

Finally, we point out that this  phenomenon  persists in the toric setting:

\begin{prop} Let $(M^4,\omega)$ be a Hamiltonian $T$-space. Then there are 
$\omega$-compatible, 
$T^2$-invariant  almost-K\"ahler  metrics $g_k$ such that $Y([g_k]) \to -\infty$.
\end{prop}
\begin{proof}
On a neighborhood of a $T^2$ orbit, we can again take
coordinates $(x,y,u,v)$ so that \eqref{darboux} holds, with $y,v$  coordinates on the base, and with 
$(x,u)$ now $(\RR/\ZZ)$-valued fiber coordinates adapted to the action. We again consider metrics $g_f$ as in \eqref{ansatz}, 
but with  $f$  now a function of $(y,v)$ only. Choosing a sequence $f_k$ of such functions
which are highly oscillatory in $v$ in a small region makes the $L^2$ norm
of the Nijenhuis tensor tend to infinity, so that 
$\int s_{g_k}~d\mu_{g_k} \to -\infty$, while the volume $[\omega]^2/2$ remains fixed.  Hence $Y([g_{f_k}])\to -\infty$, as claimed. 
\end{proof}


\begin{thebibliography}{10}

\bibitem{abreu}
{\sc M.~Abreu}, {\em K\"ahler metrics on toric orbifolds}, J. Differential
  Geom., 58 (2001), pp.~151--187.

\bibitem{apodrag}
{\sc V.~Apostolov and T.~Dr{\u{a}}ghici}, {\em The curvature and the
  integrability of almost-{K}\"ahler manifolds: a survey}, in Symplectic and
  contact topology: interactions and perspectives ({T}oronto, {ON}/{M}ontreal,
  {QC}, 2001), vol.~35 of Fields Inst. Commun., Amer. Math. Soc., Providence,
  RI, 2003, pp.~25--53.

\bibitem{AHS}
{\sc M.~F. Atiyah, N.~J. Hitchin, and I.~M. Singer}, {\em Self-duality in
  four-dimensional {R}iemannian geometry}, Proc. Roy. Soc. London Ser. A, 362
  (1978), pp.~425--461.

\bibitem{bama}
{\sc S.~Bando and T.~Mabuchi}, {\em Uniqueness of {E}instein {K}\"ahler metrics
  modulo connected group actions}, in Algebraic geometry, {S}endai, 1985,
  vol.~10 of Adv. Stud. Pure Math., North-Holland, Amsterdam, 1987, pp.~11--40.

\bibitem{bpv}
{\sc W.~Barth, C.~Peters, and A.~Van~de Ven}, {\em Compact complex surfaces},
  vol.~4 of Ergebnisse der Mathematik und ihrer Grenzgebiete (3),
  Springer-Verlag, Berlin, 1984.

\bibitem{biqdon}
{\sc O.~Biquard}, {\em M\'etriques k\"ahl\'eriennes extr\'emales sur les
  surfaces toriques (d'apr\`es {S}. {D}onaldson)}, Ast\'erisque,  (2011),
  pp.~Exp. No. 1018, viii, 181--201.
\newblock S{\'e}minaire Bourbaki. Vol. 2009/2010. Expos{\'e}s 1012--1026.

\bibitem{blair}
{\sc D.~E. Blair}, {\em The ``total scalar curvature'' as a symplectic
  invariant and related results}, in Proceedings of the 3rd Congress of
  Geometry (Thessaloniki, 1991), Thessaloniki, 1992, Aristotle Univ.
  Thessaloniki, pp.~79--83.

\bibitem{calabix}
{\sc E.~Calabi}, {\em Extremal {K}\"ahler metrics}, in Seminar on Differential
  Geometry, vol.~102 of Ann. Math. Studies, Princeton Univ. Press, Princeton,
  N.J., 1982, pp.~259--290.

\bibitem{calabix2}
\leavevmode\vrule height 2pt depth -1.6pt width 23pt, {\em Extremal {K}\"ahler
  metrics. {II}}, in Differential {G}eometry and {C}omplex {A}nalysis,
  Springer, Berlin, 1985, pp.~95--114.

\bibitem{chenlebweb}
{\sc X.~Chen, C.~LeBrun, and B.~Weber}, {\em On conformally {K}\"ahler,
  {E}instein manifolds}, J. Amer. Math. Soc., 21 (2008), pp.~1137--1168.

\bibitem{xxel}
{\sc X.~X. Chen}, {\em Space of {K}\"ahler metrics. {III}. {O}n the lower bound
  of the {C}alabi energy and geodesic distance}, Invent. Math., 175 (2009),
  pp.~453--503.

\bibitem{xxgang2}
{\sc X.~X. Chen and G.~Tian}, {\em Geometry of {K}\"ahler metrics and
  foliations by holomorphic discs}, Publ. Math. Inst. Hautes \'Etudes Sci.,
  (2008), pp.~1--107.

\bibitem{delpezzo}
{\sc M.~Demazure}, {\em Surfaces de del {P}ezzo, {II}, {III}, {IV}, {V}}, in
  S\'eminaire sur les {S}ingularit\'es des {S}urfaces, vol.~777 of Lecture
  Notes in Mathematics, Berlin, 1980, Springer, pp.~21--69.

\bibitem{derd}
{\sc A.~Derdzi{\'n}ski}, {\em Self-dual {K}\"ahler manifolds and {E}instein
  manifolds of dimension four}, Compositio Math., 49 (1983), pp.~405--433.

\bibitem{donfield}
{\sc S.~K. Donaldson}, {\em Remarks on gauge theory, complex geometry and
  {$4$}-manifold topology}, in Fields {M}edallists' lectures, vol.~5 of World
  Sci. Ser. 20th Century Math., World Sci. Publ., River Edge, NJ, 1997,
  pp.~384--403.

\bibitem{dontor}
\leavevmode\vrule height 2pt depth -1.6pt width 23pt, {\em Scalar curvature and
  stability of toric varieties}, J. Differential Geom., 62 (2002),
  pp.~289--349.

\bibitem{fujukisug}
{\sc A.~Fujiki}, {\em Moduli space of polarized algebraic manifolds and
  {K}\"ahler metrics [translation of {S}\^ugaku {\bf 42} (1990), no. 3,
  231--243; {MR}1073369 (92b:32032)]}, Sugaku Expositions, 5 (1992),
  pp.~173--191.
\newblock Sugaku Expositions.

\bibitem{fultor}
{\sc W.~Fulton}, {\em Introduction to toric varieties}, vol.~131 of Annals of
  Mathematics Studies, Princeton University Press, Princeton, NJ, 1993.

\bibitem{fuma0}
{\sc A.~Futaki and T.~Mabuchi}, {\em Bilinear forms and extremal {K}\"ahler
  vector fields associated with {K}\"ahler classes}, Math. Ann., 301 (1995),
  pp.~199--210.

\bibitem{fuma}
\leavevmode\vrule height 2pt depth -1.6pt width 23pt, {\em Moment maps and
  symmetric multilinear forms associated with symplectic classes}, Asian J.
  Math., 6 (2002), pp.~349--371.

\bibitem{kirbyperiod}
{\sc D.~T. Gay and R.~Kirby}, {\em Constructing symplectic forms on 4-manifolds
  which vanish on circles}, Geom. Topol., 8 (2004), pp.~743--777 (electronic).

\bibitem{GH}
{\sc P.~Griffiths and J.~Harris}, {\em Principles of Algebraic Geometry},
  Wiley-Interscience, New York, 1978.

\bibitem{guiltor}
{\sc V.~Guillemin}, {\em Moment maps and combinatorial invariants of
  {H}amiltonian {$T^n$}-spaces}, vol.~122 of Progress in Mathematics,
  Birkh\"auser Boston Inc., Boston, MA, 1994.

\bibitem{gursky}
{\sc M.~J. Gursky}, {\em The {W}eyl functional, de {R}ham cohomology, and
  {K}\"ahler-{E}instein metrics}, Ann. of Math. (2), 148 (1998), pp.~315--337.

\bibitem{kohonda}
{\sc K.~Honda}, {\em Transversality theorems for harmonic forms}, Rocky
  Mountain J. Math., 34 (2004), pp.~629--664.

\bibitem{inyoungthesis}
{\sc I.~Kim}, {\em Almost-K\"ahler Anti-Self-Dual Metrics}, PhD thesis, State
  University of New York at Stony Brook, 2013.

\bibitem{jongsuak}
{\sc J.~Kim}, {\em A closed symplectic four-manifold has almost {K}\"ahler
  metrics of negative scalar curvature}, Ann. Global Anal. Geom., 33 (2008),
  pp.~115--136.

\bibitem{jongsuak2}
\leavevmode\vrule height 2pt depth -1.6pt width 23pt, {\em Almost {K}\"ahler
  metrics of negative scalar curvature on symplectic manifolds}, Math. Z., 262
  (2009), pp.~381--388.

\bibitem{okobweyl}
{\sc O.~Kobayashi}, {\em On a conformally invariant functional of the space of
  {R}iemannian metrics}, J. Math. Soc. Japan, 37 (1985), pp.~373--389.

\bibitem{lalmcd}
{\sc F.~Lalonde and D.~McDuff}, {\em {$J$}-curves and the classification of
  rational and ruled symplectic {$4$}-manifolds}, in Contact and symplectic
  geometry ({C}ambridge, 1994), vol.~8 of Publ. Newton Inst., Cambridge Univ.
  Press, Cambridge, 1996, pp.~3--42.

\bibitem{lebhem}
{\sc C.~LeBrun}, {\em Einstein metrics on complex surfaces}, in Geometry and
  {P}hysics ({A}arhus, 1995), vol.~184 of Lecture Notes in Pure and Appl.
  Math., Dekker, New York, 1997, pp.~167--176.

\bibitem{lcp2}
\leavevmode\vrule height 2pt depth -1.6pt width 23pt, {\em Yamabe constants and
  the perturbed {S}eiberg-{W}itten equations}, Comm. An. Geom., 5 (1997),
  pp.~535--553.

\bibitem{lsymp}
\leavevmode\vrule height 2pt depth -1.6pt width 23pt, {\em Einstein metrics,
  symplectic minimality, and pseudo-holomorphic curves}, Ann. Global Anal.
  Geom., 28 (2005), pp.~157--177.

\bibitem{lebuniq}
\leavevmode\vrule height 2pt depth -1.6pt width 23pt, {\em On {E}instein,
  {H}ermitian 4-manifolds}, J. Differential Geom., 90 (2012), pp.~277--302.

\bibitem{lebtoric}
\leavevmode\vrule height 2pt depth -1.6pt width 23pt, {\em Calabi energies of
  extremal toric surfaces}, in Surveys in Differential Geometry, Vol. XVIII:
  Geometry and Topology, Int. Press, Boston, MA, Boston, 2013, pp.~195--226.

\bibitem{lebhem10}
\leavevmode\vrule height 2pt depth -1.6pt width 23pt, {\em Einstein manifolds
  and extremal {K}\"ahler metrics}, J. Reine Angew. Math., 678 (2013),
  pp.~69--94.

\bibitem{lejmi}
{\sc M.~Lejmi}, {\em Extremal almost-{K}\"ahler metrics}, Internat. J. Math.,
  21 (2010), pp.~1639--1662.

\bibitem{lichconf}
{\sc A.~Lichnerowicz}, {\em Transformation infinit\'esimales conformes de
  certaines vari\'et\'es riemanniennes compactes}, C. R. Acad. Sci. Paris, 241
  (1955), pp.~726--729.

\bibitem{liu1}
{\sc A.-K. Liu}, {\em Some new applications of general wall crossing formula,
  {G}ompf's conjecture and its applications}, Math. Res. Lett., 3 (1996),
  pp.~569--585.

\bibitem{liu}
{\sc K.~Liu}, {\em Geometric height inequalities}, Math. Res. Lett., 3 (1996),
  pp.~693--702.

\bibitem{cubic}
{\sc Y.~I. Manin}, {\em Cubic {F}orms: {A}lgebra, {G}eometry, {A}rithmetic},
  North-Holland Publishing Co., Amsterdam, 1974.
\newblock Translated from the Russian by M. Hazewinkel.

\bibitem{morgan}
{\sc J.~Morgan}, {\em The {S}eiberg-{W}itten Equations and Applications to the
  Topology of Smooth Four-Manifolds}, vol.~44 of Mathematical Notes, Princeton
  University Press, 1996.

\bibitem{sunspot}
{\sc Y.~Odaka, C.~Spotti, and S.~Sun}, {\em Compact moduli spaces of {D}el
  {P}ezzo surfaces and {K}\"ahler-{E}instein metrics}.
\newblock e-print arXiv:1210.0858 [math.DG], 2012.

\bibitem{page}
{\sc D.~Page}, {\em A compact rotating gravitational instanton}, Phys. Lett.,
  79B (1979), pp.~235--238.

\bibitem{perutz}
{\sc T.~Perutz}, {\em Zero-sets of near-symplectic forms}, J. Symplectic Geom.,
  4 (2006), pp.~237--257.

\bibitem{poonthesis}
{\sc Y.~S. Poon}, {\em Compact self-dual manifolds with positive scalar
  curvature}, J. Differential Geom., 24 (1986), pp.~97--132.

\bibitem{s}
{\sc Y.~T. Siu}, {\em The existence of {K}\"ahler-{E}instein metrics on
  manifolds with positive anticanonical line bundle and a suitable finite
  symmetry group}, Ann. of Math. (2), 127 (1988), pp.~585--627.

\bibitem{taubes}
{\sc C.~H. Taubes}, {\em The {S}eiberg-{W}itten invariants and symplectic
  forms}, Math. Res. Lett., 1 (1994), pp.~809--822.

\bibitem{taubes2}
\leavevmode\vrule height 2pt depth -1.6pt width 23pt, {\em More constraints on
  symplectic forms from {S}eiberg-{W}itten invariants}, Math. Res. Lett., 2
  (1995), pp.~9--14.

\bibitem{taubes3}
\leavevmode\vrule height 2pt depth -1.6pt width 23pt, {\em The
  {S}eiberg-{W}itten and {G}romov invariants}, Math. Res. Lett., 2 (1995),
  pp.~221--238.

\bibitem{taubescirc}
\leavevmode\vrule height 2pt depth -1.6pt width 23pt, {\em A proof of a theorem
  of {L}uttinger and {S}impson about the number of vanishing circles of a
  near-symplectic form on a 4-dimensional manifold}, Math. Res. Lett., 13
  (2006), pp.~557--570.

\bibitem{tian}
{\sc G.~Tian}, {\em On {C}alabi's conjecture for complex surfaces with positive
  first {C}hern class}, Invent. Math., 101 (1990), pp.~101--172.

\bibitem{yano}
{\sc K.~Yano}, {\em Differential geometry on complex and almost complex
  spaces}, International Series of Monographs in Pure and Applied Mathematics,
  Vol. 49, A Pergamon Press Book. The Macmillan Co., New York, 1965.

\end{thebibliography}

\end{document}